\documentclass{article}
 \usepackage{euscript}
\usepackage{amsmath}
\usepackage{graphicx}
\usepackage{amsthm}
\usepackage{amssymb}
\usepackage{epsfig}
\usepackage[all]{xy}
\usepackage{url}
\usepackage{tkz-berge}
\theoremstyle{theorem}
\newtheorem{theorem}{Theorem}[section]
\newtheorem{corollary}[theorem]{Corollary}

\newtheorem{lemma}[theorem]{Lemma}
\newtheorem{proposition}[theorem]{Proposition}

\newtheorem{question}{Question}[section]
\numberwithin{equation}{section}
\title{Some existence problems regarding partial Latin squares}
\author{Masood Aryapoor\footnote{E-mail: aryapoor2002@yahoo.com}}
\date{}

\begin{document}

\maketitle

 \begin{abstract}

Latin squares are interesting combinatorial objects with many applications.
When working with Latin squares, one is sometimes led to deal with partial Latin squares, a generalization of Latin squares. 
One of the problems regarding partial Latin square and with applications to Latin squares is whether a partial Latin square with a given set of conditions exists.
The goal of this article is to introduce some problems of this kind and answer some existence questions regarding  partial Latin squares. 

\end{abstract}  


\begin{section}{Introduction}

A \textit{partial Latin square} (or \textbf{PLS} for short) $P$ is a finite nonempty subset of $\mathbb{N}^3=\mathbb{N}\times\mathbb{N}\times \mathbb{N}$ for which the restriction maps  $Pr_{ij}: P\to  \mathbb{N}^2$ are 
injective for $1\leq i<j \leq 3$. Here $Pr_{ij}: \mathbb{N}^3\to \mathbb{N}^2$ is the projection map on the $(i,j)$th factor. A partial Latin square can be represented by using an array  
in the following way. Consider an array whose rows and columns  are indexed by  natural numbers. To the $(i,j)$th cell of the array, assign $k$ if   $(i,j,k)\in P$, and let it remain empty if 
no such $k$ exists. The resulting array, denoted by $A(P)$, has the following properties: it has only a finitely many nonempty cells and every natural number appears at most once in each row and each column 
of $A(P)$. It is easy to see that $P\mapsto A(P)$ gives a 1-1 correspondence between  the set of partial Latin squares and the set of arrays having the mentioned properties. 
Similarly $P$ can also be represented on finite arrays. In this representation, the entries of the cells are usually called the \textit{symbols} of $P$.   
  
Given a partial Lain square $P$, we can associate some parameters to it. The first parameter is the number of elements of $P$ which is called the \textit{volume} of $P$ and denoted by $v(P)$. 
Put $R(P)=Pr_1(P)$, $C(P)=Pr_2(P)$ and $S(P)=Pr_3(P)$ where 
$Pr_i: \mathbb{N}^3\to \mathbb{N}$ is the projection map on the $i$th factor. The number $r(P)=|R(P)|$ is called the \textit{number of rows} of $P$ where $|X|$ stands for the cardinality of a set  $X$. 
Similarly $c(P) =|C(P)|$ is called the \textit{number of columns} of $P$ and $s(P) =|S(P)|$ is called the \textit{number of symbols} of $P$. To get more parameters for $P$,
let $R(P)$ consist of natural numbers $i_1<i_2<...<i_{r(P)}$. Then we obtain natural numbers $|Pr_1^{-1}(i)\cap P|$ for $i=i_1,...,i_{r(P)}$. These natural numbers are called the \textit{row-parameters}
of $P$. In a similar way, the \textit{column-parameters}  and \textit{symbol-parameters} of $P$ are defined. 

The question handled in  this paper is the following. 

\begin{question}\label{QUESTION}

Suppose that natural numbers $m_1,...,m_r$, $n_1,...,n_c$ and $p_1,...,p_s$ are given. How can one decide if there is a partial Latin square  $P$ having 
row-parameters $m_1,...,m_r$, column-parameters $n_1,...,n_c$ and symbol-parameters $p_1,...,p_s$?

\end{question}  
  
A remark about this question is in order. One can easily derive some necessary conditions on $m_1,...,m_r, n_1,...,n_c, p_1,...,p_s$ for the existence of such a PLS. The author is  unaware
if a ''reasonable" set of necessary and sufficient conditions exists in the literature. In any case,  this question is partly answered in this paper.


\end{section}

\begin{section}{Existence of partial Latin squares}

Before tackling Question \ref{QUESTION}, we need the following lemma from Graph Theory, see \cite{BM} for the relevant material in Matching Theory.    

\begin{lemma}  \label{MATCHINGLEMMA}

Suppose that $G=(X,Y)$ is a bipartite graph such that the degree of each vertex in $G$ is less than or equal to a given natural number $n$. Suppose that 
$X_1\subset X$ and $Y_1\subset Y$ are two sets of vertices such that $d_G(z)=n$ for  all $z\in X_1\cup Y_1$. Then $G$ has  a matching covering all the vertices in 
$X_1\cup Y_1$. 

\end{lemma}

\begin{proof}

First we show that $G$ has a matching $M$ covering all the vertices in $X_1$. In fact for every subset $Z\subset X_1$, we have 
$n|Z|=\sum_{z\in N_G(Z)} d(z)\leq n|N_G(Z)|$, i.e. $|Z| \leq  |N_G(Z)|$  where $N_G(Z)$ is the set of vertices in $G$ which are adjacent to some vertex in $Z$.   
By Hall's theorem, $G$ has a matching $M$ which covers $X_1$. Similarly $G$ has a matching $N$ which covers $Y_1$. By deleting some edges if necessary,  we can furthermore assume that $M$ has $|X_1|$ edges 
and $N$ has $|Y_1|$ edges. Let $M\Delta N$ be the symmetric difference of $M$ and $N$. It is known (and in fact easy to see) that   $M\Delta N$ is a vertex-disjoint union of cycles and 
paths. We construct a matching $K\subset M\cup N$ covering all the vertices in $X_1\cup Y_1$ in some steps. 

Given a cycle $C$ in $M\Delta N$, we put $K_C$ to be the set of edges of $C$ which belong to $M$. Clearly $K_C$ covers all the vertices of $C$. 

Next suppose that $P=v_1,...,v_m$ is a maximal path in  $M\Delta N$ with edges $v_1v_2\in M,v_2v_3\in N,...$. Since vertex $v_2$ is covered by both $M$ and $N$, we have  $v_2\in X_1\cup Y_1$.
W consider two cases depending on whether $v_2\in X_1$ or $v_2\in Y_1$. First suppose that $v_2\in X_1$. Then $v_1\notin Y_1$, since otherwise there would exist some vertex $x$ such that $xv_1\in N$, 
a contradiction to the fact that $P$ is a maximal path in $M\Delta N$. 
It is now easy to see that we must have $v_3\in Y_1, v_4\in X_1,...$. If $v_m\in X_1$ (i.e. $m$ is even ), then set $K_P$ to be the set of edges of $P$ used in $M$. If $v_m\in Y_1$ (i.e. $m$ is odd), then put $K_P$ to be 
 the set of edges of $P$ used in $N$. Either way, it can be seen that $K_P$ covers all the vertices of $P$ belonging to $X_1\cup Y_1$.  
Now consider the second case, i.e. $v_2\in Y_1$. 
Then we must clearly have $v_1\in X_1$.  In this case put $K_P$ to be the set of edges of $P$ used in $M$. Then $K_P$ covers all the vertices of $P$ belonging to $X_1\cup Y_1$. To see this, note that  
either  $m$ is odd in which case $v_3\in X_1,v_4\in Y_1,...v_{m-1}\in Y_1, v_m\in X\setminus X_1$,
or $m$ is even in which case   $v_3\in X_1,v_4\in Y_1,...v_{m-1}\in X_1, v_m\in Y\setminus Y_1$.

Similarly we define $K_P$ where   $P=v_1,...,v_m$ is a maximal path in  $M\Delta N$ with edges $v_1v_2\in N,v_2v_3\in M,...$.

Now define $K$ to be the following set of edges of $G$, $K=(M\cap N)\cup (\cup_{Q} K_Q)$ where $Q$ ranges over the set of cycles and maximal paths in $M\Delta N$. I claim that $K$ 
is a matching covering  all the vertices in $X_1\cup Y_1$. First we prove that $K$ is a matching. In the way we have defined $K_Q$'s, it is clear that no vertex is covered by more than one edge in  
$\cup_{Q} K_Q$. It is also clear that $M\Delta N$ is a matching. Finally, since $M\cap N$ and  $\cup_{Q} K_Q$ have no vertex in common, we see that $K$ is in fact a matching. 
Since every vertex of $X_1\cup Y_1$ belongs to $M\cap N$ or one of the cycles or paths of $M\Delta N$, we see that every vertex of 
$X_1\cup Y_1$ is covered by some edge of $K$, as demonstrated above when defining $K_Q$'s.

\end{proof}

\begin{subsection}{Special cases of Question \ref{QUESTION}}

We start with a useful lemma. 

\begin{lemma}  \label{MAINLEMAA}

Let $B$ be a nonempty set of $v$ cells of an $r\times c$ array. Suppose that $B$ has $n_i>0$ cells in the $i$th row and $m_j>0$ cells in the $j$th row for each $i$ and $j$. 
Then the cells in $B$ can be filled out with natural numbers in such a way that we obtain   a PLS, $P$  with $s(P)=\max(n_1,...,n_r,m_1,...,m_c)$.

\end{lemma}

\begin{proof} 

Proof by induction on $t=\max(n_1,...,n_r,m_1,...,m_c)$. If $t=1$, then it implies that $n_i=1$ and $m_j=1$ for all $i,j$ which means $B$ has exactly one cell in each row and one cell in  
each column and consequently, we can easily construct the desired PLS, $P$ with just one symbol.  

Now suppose that a natural number $p$ is given and the lemma holds for all natural numbers $t<p$. We need to prove the lemma for $t=p$. Without loss of generality, we can assume that 
$n_r\leq n_{r-1}\leq ...\leq n_1$ with $n_1=\cdots=n_{r_1}=p$ but $n_{r_1+1}< p$. Similarly we can assume that    
$m_c\leq m_{c-1}\leq ...\leq m_1$ with $m_1=\cdots=m_{c_1}=p$ but $m_{c_1+1}< p$.

Now consider the following bipartite graph $G$. The set of vertices  of $G$ is the union of $X=\{1,...,r\}$ and $Y=\{1,...,c\}$. The vertex $x\in X$ is adjacent to $y\in Y$
if cell $(x,y)$ of the array belongs to $B$. Setting $X_1=\{1,...,r_1\}$ and $Y_1=\{1,...,c_1\}$, we can apply Lemma \ref{MATCHINGLEMMA} to obtain a matching 
$K$ of $G$ covering all the vertices in  $X_1\cup Y_1$. Let the edges of the matching correspond to cells $(i_1,j_1),...,(i_k,j_k)$. Set $B'=B\setminus \{(i_1,j_1), ..., (i_k,j_k)\} .$
 
It is now easy to see that no row or column of the array can have more than $p-1$ cells belonging to $B'$. 
However note that the first row or the first colum has $p-1$ cells belonging to $B'$. So, by induction, 
we can construct a PLS on $B'$ with symbols  $1,...,p-1$. Now if we fill out the remaining cells of $B$ with $p$, then it can easily be seen that 
we have a PLS on $B$ with exactly $p$ symbols.

\end{proof}

The most general form  of Question \ref{QUESTION}, answered in this paper, is the following.

\begin{theorem}\label{MAINTHEOREM}

Suppose that natural numbers $n_1,...,n_r$, $m_1,...,m_c$ and $s$ are given. Then there is a PLS,  $P$  having  row-parameters $n_1,...,n_r$ and column-parameters
$m_1,...,m_c$ such that $s(P)=s$ if and only if  the following hold: (1) $n_1+\cdots+n_r=m_1+\cdots+m_c=v$.
(2) For subsets $I\subset \{1,...,r\}$ and $J\subset  \{1,...,c\}$ we have $\sum_{i\in I} n_i+\sum_{j\in J}m_j\leq v+|I| |J|$.  
(3) $\max(n_1,...,n_r,m_1,...,m_c)\leq s\leq v$.    
  
\end{theorem}

\begin{proof}

First suppose that such  a PLS, $P$ exists. Then it is clear that the first condition holds where $v$ is just the volume of $P$. To see the second condition, consider an $r\times s$ matrix $E$ where 
$E_{ij}=1$ if cell $(i,j)$ belongs to $P$ and $E_{ij}=0$ otherwise. The well-known criteria of the Gale-Ryser theorem, see  \cite{BR} for example,  gives the condition (2). Finally, we see that $v$, the volume of $P$, 
is at least the number $s$ of the symbols of $P$ and the number of symbols $s$ cannot be less than the number of cells of $P$ in some row or column. Therefore   (3) must hold.      

Conversely, suppose that conditions (1), (2) and (3)  hold. According to the Gale-Ryser theorem, the first two conditions imply that there is a $(0,1)$-matrix $E$ whose row-sum vector is 
$(n_1,...,n_r)$ and whose column-sum vector is $(m_1,...,m_c)$.  Consider the following set $B$ of cells of an $r\times c$ array. Cell $(i,j)$ belongs to $B$ if and only if $E_{ij}=1$. 
It is immediate that $B$ has $n_i$ cells in row $i$ and $m_j$ cells in column $j$ for every $i,j$. By Lemma  \ref{MAINLEMAA}, there is a PLS, $Q$ on $B$ with exactly 
$s_0= \max(n_1,...,n_r,m_1,...,m_c)$ symbols. Let the symbols be $1,...,s_0$. Choose $s-s_0$ arbitrary cells of $Q$ and change their symbols to $s_0+1,...,s$ in an arbitrary order such that each symbol 
$s_0+1,...,s$ is used exactly once.  This is possible since $s_0\leq s\leq v$. The result is now a PLS having the desired conditions.

\end{proof}

Another special case of Question \ref{QUESTION} is given below.

\begin{proposition} \label{PROPOSITION}

Suppose that natural numbers $n_1,...,n_r$, $c$ and $s$ are given. Then there is a PLS,  $P$  having row-parameters $n_1,...,n_r$ such that $c(P)=c$ and $s(P)=s$,  
if and only if  $\max(c,s)\leq n_1+\cdots +n_r\leq c s  $ and $n_i\leq \min (c,s)$ for every $i=1,...,r$. 
 
\end{proposition}

\begin{proof}

First suppose that such  a PLS, $P$ exists. Since $n_1+\cdots +n_r$ is the volume of $P$ and each column has at least one cell in $P$, we see that
 $c\leq n_1+\cdots +n_r$. Similarly, we have  $s\leq n_1+\cdots +n_r$. Since $P$ is a PLS with $c(P)=c$ and $s(P)=s$, its volume   $n_1+\cdots +n_r$ is at most $s t$. It is clear that a row of the array 
cannot have more than $c$ cells in $P$ and it cannot have more than $s$ cells of the array. In other words, we have $n_i\leq \min (c,s)$ for all $i$.

Conversely, suppose that the conditions hold. Without loss of generality we assume that $c\leq s$. Choose a set $B$ of cells in an $r\times c$ array where $B$ has exactly $n_i$ cells of the array in row $i$ for every $i$. 
This is possible since $n_i\leq c$ for every $i=1,...,r$. For every $j=1,...,c$, let $p_j$ be the number of cells of $B$ in the
$j$th column. Suppose that one of numbers $p_1,...,p_c$, say $p_1$, is greater than $s$. Since $p_1+\cdots+p_c=n_1+\cdots+n_r\leq  c s $, we see that there is some $p_j$ with $p_j< s$. Now, since $p_j<p_1$, there must 
exist  $1\leq i\leq r$ such that $(i,1)\in B$ but $(i,j)\notin B$. Set $B_1=(B\setminus\{(i,1)\})\cup \{(i,j)\}$. It is easy to see that $B_1$ has exactly $n_k$ cells in each row $k$ for every $k=1,...,r$ and 
has exactly $p_1-1,p_2,...,p_{j-1},p_j+1,p_{j+1},...,p_c$ cells in columns $1,...,c$ respectively. Continuing this process, we obtain a subset $B'$ of cells of the array with $n_i$ cells in row $i$    
and $m_j\leq s$ cells in column $j$ for each $i$ and $j$. Now it is clear that $n_1,...,n_r$ and $m_1,...,m_c$ and $s$ satisfy the conditions in 
Theorem \ref{MAINTHEOREM}, and therefore there is a PLS, $P$   having row-parameters $n_1,...,n_r$, column parameters $m_1,...,m_c$ such that $s(P)=s$.
It implies that $P$ has row-parameters $n_1,...,n_r$ and we have $c(P)=c$, $s(P)=s$.

\end{proof}

The following case of Question \ref{QUESTION}, is the last case treated in this paper.  

\begin{corollary} 

Suppose that natural numbers $r$, $c$, $s$ and $v$ are given. Then there is a PLS,  $P$  with $r(P)=r$, $c(P)=c$, $s(P)=s$ and $v(P)=v$ 
if and only if  $\max(r,c,s)\leq v\leq \min(r c, c s, r s) $. 
 
\end{corollary}

\begin{proof}

First suppose that such  a PLS, $P$ exists. Then $P$ has one cell in each row which means $r\leq v$. Similarly one can show that $c\leq v$ and  $s\leq v$. Since $P$
can be represented on an $r\times c$ array and $v$ is the number of cells of the array occupied by $P$, it is immediate that $v\leq r c $. Similarly we have $v\leq c s$ and $v\leq r s$.

Conversely, suppose that the inequalities hold. Choose a set $B$ of cells of an $r\times c$ array such that $|B|=v$. This is possible since $v\leq r c $.  
Following the same argument as in the proof of Proposition \ref{PROPOSITION}, by starting from $B$ and using the condition $v\leq r s$, we can construct a set $B'$ of cells in the array 
such that  $B'$ has $n_i\leq s$ cells in the $i$th row for every $i=1,...,r$. It is obvious that  $n_i\leq c$ cells in the $i$th row for every $i=1,...,r$. Now natural numbers 
$n_1,...,n_r$, $c$ and $s$ satisfy the conditions in Proposition \ref{PROPOSITION}. Therefore there is a PLS, $P$ having row-parameters $n_1,...,n_r$ such that $c(P)=c$ and $s(P)=s$.
It is clear that $P$ is the desired PLS and therefore the proof is complete.

\end{proof}

\end{subsection}

\end{section}

\end{document}